\documentclass[a4paper]{article}
\usepackage[affil-it]{authblk}                    
\usepackage{graphicx}
\usepackage{cite}
\usepackage{amsfonts}
\usepackage{amsmath}
\usepackage{amsthm}
\usepackage[english]{babel}
\newtheorem{theorem}{Theorem}
\newtheorem{corollary}{Corollary}[theorem]

\newtheorem{definition}{Definition}[section]
\newtheorem{example}{Example}[section]
\newtheorem{lemma}{Lemma}[section]
\usepackage{graphicx}
\usepackage{cite}
\usepackage{amsfonts}
\usepackage{amsmath}
\usepackage{epstopdf}
\usepackage[english]{babel}
\usepackage{graphicx}
\usepackage{float}
\usepackage{mleftright}
\begin{document}

\title{Finding the set of global minimizers of a piecewise affine function
}


\author{Majid E. Abbasov         
}


\affil{
              St. Petersburg State University, SPbSU,\\ 7/9 Universitetskaya nab., St.
              Petersburg, 199034 Russia. \\

              Institute for Problems in Mechanical
              Engineering of the RAS\\
              61, Bolshoj pr. V.O., St. Petersburg,
199178\\
              {m.abbasov@spbu.ru, abbasov.majid@gmail.com}           
}


\maketitle

\begin{abstract}
In the present work we study a problem of finding a global minimum
of a piecewise affine function. We employ optimality conditions
for the problem in terms of coexhausters and use them to state and
prove necessary and sufficient conditions for a piecewise affine
function to be bounded from below. We construct a simple method
based on these conditions which allows one to get the minimum
value of a studied function and the corresponding set of all its
global minimizers. These results are built via coexhauster notion.
This notion was introduced by V.~F.~Demyanov.

Coexhausters are families of convex compact sets that allow one to
represent the approximation of the increment of the studied
function at a considered point in the form of minmax or maxmin of
affine functions. We take these representations as a definition of
a piecewise affine function and show that they correspond with the
definitions for piecewise affine function given by other
researchers. All the conditions and methods were obtained by means
of coexhausters theory. In the paper we give some necessary facts
from this theory.

A lot of illustrative numerical examples are provided throughout
the paper.

\end{abstract}

\section*{Introduction}
\label{abbasov_sec_intro}

Piecewise affine functions are important for different areas of
mathematics and its applications. For example they are used in DC
programming \cite{Gaudioso2017}, optimal control \cite{Dempe2019},
global optimization \cite{Mangasarian2005, Zhang2008},
approximation problems \cite{Misener2010,Silva2014} and so on. We
will consider piecewise affine functions from the coexhausters
theory point of view. Coexhausters are effective tool for the
study of nondifferentiable functions and are important objects in
so called constructive nonsmooth analysis
\cite{Demyanov-Pardalos2014}.

Local behavior of any smooth function at a point is provided by
the linear function which is described by means of the gradient.
Therefore optimality conditions for the smooth functions as well
as techniques for building directions of descent and ascent are
expressed via gradient. Obviously one need a different form of
approximation of the increment of a studied function to follow the
same idea in the nonsmooth case. One of the ways here is to
consider this approximation in the form of minmax or maxmin of
affine functions. Coexhausters are families of convex compact sets
that provide this type of representation for the approximation of
the increment. This notion was proposed by Demyanov in
\cite{Demyanov-dem00a,Demyanov-dem991}. The well-developed
calculus of coexhausters lets one to build these families for a
wide class of functions. Optimality conditions in terms of
coexhausters as well as technique for obtaining directions of
descent and ascent in cases when these conditions are not
satisfied were derived in \cite{Abbasov2013, Demyanov_dem2012,
Abbasov2019}. This led to the emergence of the  coexhausters based
optimization algorithms.

Thus piecewise affine functions are involved in the coexhausters
theory. Therefore, it is not surprising that exactly this theory
allowed us to obtain important new results for such functions.

The paper is organized as follows. In Section 1 we discuss
coexhausters notion and the definition of piecewise affine
function. In section 2 we describe optimality conditions in terms
of coexhausters.
The advantage of the optimality conditions of this form is
disclosed subsequently, as all the proposed results and methods
are based on these conditions. Since we consider a problem of
finding a global minimum, it is natural to deal with functions
which are bounded below. We derive the condition which is
equivalent to the lower boundedness. In Section 3 we describe a
simple technique for finding the minimum value of the function.
The method for obtaining the set of all global minimizers is
developed in Section 4. Throughout the paper we provide numerical
examples to illustrate how proposed results can be applied.

\section{Coexhausters and piecewise affine functions} \label{abbasov_sec1}

Let $X\subset{\mathbb{R}^{n}}$ be an open set and a function
$f\colon X\to \mathbb{R}$ be given. Let the function $f$ be
continuous at a point $x\in X$. We say that at the point $x$ the
function $f$ has an upper coexhauster if the following expansion
holds:
\begin{equation}\label{1904e3a}
f(x+\Delta)=f(x)+\min_{C\in{\overline{E}(x)}}\max_{[a,v]\in{C}}[a+\langle
v,\Delta\rangle]+o_{x}(\Delta)
\end{equation}
where $\overline{E}(x)$ is a family of convex compact sets in
${\mathbb{R}}^{n+1}$, and
\begin{equation}\label{1904e11a}
\lim_{\alpha\downarrow0}\frac{o_{x}(\alpha\Delta)}{\alpha}=0\quad
\forall\Delta\in{\mathbb{R}}^n.
\end{equation}

The set $\overline{E}(x)$ is called an upper coexhauster of $f$ at
the point $x$.

We say that at the point $x$ the function $f$ has a lower
coexhauster if the following expansion holds:
\begin{equation}\label{1904e3b}f(x+\Delta)=f(x)+\max_{C\in{\underline{E}(x)}}\min_{[b,w]\in{C}}[b+\langle w,\Delta\rangle]+o_{x}(\Delta),
\end{equation}
where $\underline{E}(x)$ is a family of convex compact sets in
${\mathbb{R}}^{n+1}$, and $o_{x}(\Delta)$ satisfies
(\ref{1904e11a}).

The set $\underline{E}(x)$ is called a lower coexhauster of the
function $f$ at the point $x$.

The function $f$ is  continuous, therefore from (\ref{1904e3a})
and (\ref{1904e3b}) (for $\Delta=0$) it follows that
 \begin{equation}\label{1904e3am}
\min_{C\in{\overline{E}(x)}}\max_{[a,v]\in{C}}a =
\max_{C\in{\underline{E}(x)}}\min_{[b,w]\in{C}}b = 0.
\end{equation}
 The notion of coexhauster was introduced in
 \cite{Demyanov-dem991,Demyanov-dem00a}. The calculus of
coexhausters that allows one to built these families for a wide
class of nonsmooth functions as well as optimality conditions in
terms of these objects were developed in
\cite{Abbasov2013,Demyanov_dem2012,Abbasov2019}.

Recall that if a set can be represented as the intersection of a
finite family of closed halfspaces it is called polyhedral
\cite{rock70}. A function $\phi\colon\mathbb{R}^n\to\mathbb{R}$ of
the form $\phi(\Delta)=a+\langle v,\Delta\rangle$, where $a\in
\mathbb{R}$, $v\in \mathbb{R}^n$ is called affine.

Now we can give a definition of a piecewise affine function (see
\cite{Eremin1998,Gorokhovik1994,Scholtes2012}).

\begin{definition} Let $\{M_j\mid j\in J\}$, $\{\phi_j(x)\mid j\in J\}$ be finite sets of polyhedral sets
and affine functions respectively, where $J=1,\dots,k$. We say
that this sets define a piecewise affine function $\Phi(x)$ on
$\mathbb{R}^n$ if the following conditions hold:

\begin{enumerate}
\item $\operatorname{int}M_j\neq\emptyset$ $\forall j\in J$,

\item $\displaystyle\bigcup_{j\in J} M_j=\mathbb{R}^n$,

\item $\operatorname{int}M_j \bigcap \operatorname{int}M_j=
\emptyset$ $i\neq j$,

\item $\Phi(x)=\phi_j(x)$ $\forall x\in M_j$, $\forall j\in J$.
\end{enumerate}
\end{definition}

Here $\operatorname{int}M$ means the interior of the set $M$.

As it was shown in \cite{Gorokhovik1994} the function $\Phi(x)$ is
piecewise affine iff it can be represented both in the forms
$$\Phi(x)=\min_{1\leq i\leq k_1}\max_{1\leq j\leq m_j}\phi_{ij}(x)$$
or
$$\Phi(x)=\max_{1\leq i\leq k_2}\min_{1\leq j\leq m_j}\phi_{ij}(x)$$
which are obviously equivalent to the representations
(\ref{1904e3a}) and (\ref{1904e3b}) when coexhausters consist of
finite numbers of convex polytopes. As we see in this case
coexhausters provide piecewise affine approximation for the
increment of the studied function at a considered point.

\section{Bounded below piecewise affine function} \label{abbasov_sec2}

Optimality conditions for the minimum most organically are stated
in terms of upper coexhauster. Therefore we deal with the
representation of a piecewise affine function of the form
$$h(x)=\min_{C\in{\overline{E}}}\max_{[a,v]\in{C}}[a+\langle
v,x\rangle],$$ where $\underline{E}$ is the family that contains
finite number of convex polytopes. If $h(0_n)=0$ (which is the
case (see \ref{1904e3am}) when $h$ is the coexhauster
approximation of increment of the studied function at a considered
point) then the following theorem gives us optimality condition
for the global minimum of the function $h$ at the origin.

\begin{theorem}\label{demyanov_opt_cond_main}
For the condition $h(x)\geq 0$ to hold for any $x\in\mathbb{R}^n$
it is necessary and sufficient that
$$C\bigcap L_0^+\neq\emptyset \quad \forall C\in\overline{E},$$
where $L_0^+=\left\{(a,0_n)\in\mathbb{R}^{n+1}\mid a\geq
0\right\}$.
\end{theorem}

This theorem was derived in \cite{Demyanov_dem2012}. Based on this
result we can formulate necessary and sufficient conditions for
the function $h$ to be bounded from below.

\begin{theorem}\label{abbasov_th1}
A function $h(x)$ is bounded from below iff it holds the condition
\begin{equation}\label{abbasov_th1_condition} C\bigcap L_0\neq\emptyset \quad \forall C\in\overline{E},\end{equation}
where $L_0=\left\{(a,0_n)\in\mathbb{R}^{n+1} \right\}$.
\end{theorem}

\begin{proof}
Let us start from necessity. Suppose that there exists $M$ such
that \mbox{$h(x)-M\geq 0$} for all $x\in\mathbb{R}^n$. This means
validity of the inequality
$$\min_{\widetilde{C}\in\widetilde{{E}}}\max_{[a,v]\in \widetilde{C}}[a+\langle v,x\rangle]\geq 0,$$
where $$\widetilde{{E}}=\left\{\widetilde{C}=C-(M,0_n)\mid
C\in\overline{E}\right\},\ C-(M,0_n)=\left\{(a-M,v)\mid (a,v)\in
C\right\}.$$ Theorem \ref{demyanov_opt_cond_main} implies that
condition $C\bigcap L_0^+\neq\emptyset$ holds for any
$\widetilde{C}\in\widetilde{E}$, whence
$$C\bigcap L_0\neq\emptyset \quad \forall C\in\overline{E}.$$

Now proceed to the sufficiency. Suppose that the condition
$C\bigcap L_0\neq\emptyset$ is true for all $C\in\overline{E}$.
Choose an arbitrary set $\widehat{C}$ from $\overline{E}$ and
denote $\displaystyle\widehat{a}=\max_{[a,0_n]\in \widehat{C}} a$,
$\displaystyle a_\ast=\min_{C\in{\overline{E}}}\max_{[a,0_n]\in
C}a$. 
The inequality
$$a_\ast=\min_{C\in{\overline{E}}}\max_{[a,0_n]\in C}a\leq \max_{[a,0_n]\in \widehat{C}} a=\widehat{a}$$
holds, therefore $(\widehat{a}-a_\ast,0_n)\in L_0^+$, whence it is
obvious that
\begin{equation}\label{abbasov_th1_suf_1}\left(\widehat{C}-(a_\ast,0)\right)\bigcap
L_0^+\neq\emptyset.\end{equation} Since $\widehat{C}$ was chosen
arbitrarily, condition (\ref{abbasov_th1_suf_1}) is valid for all
$C\in \overline{E}$
$$\left(C-(a_\ast,0)\right)\bigcap L_0^+\neq\emptyset \quad \forall C\in \overline{E}.$$
Theorem \ref{demyanov_opt_cond_main} implies the inequality
$$\min_{C\in \overline{E}}\max_{[a,v]\in C}[a-a_\ast+\langle v,x\rangle]\geq 0 \quad \forall
x\in\mathbb{R}^n,$$ which yields that $h(x)\geq a_\ast$ for all
$x\in\mathbb{R}^n$. \qed
\end{proof}

The following example demonstrates how Theorem \ref{abbasov_th1}
works.

\begin{example}\label{abbasov_exmpl1}
Let the function $h_1\colon\mathbb{R}\to\mathbb{R}$ is given by
$$h_1(x)=\min\left\{\max\left\{2+x,\frac{1}{2}x+1\right\}, \max\left\{-2+x,-x\right\}\right\}.$$
The family $\overline{E}=\{C_1,C_2\}$, where
$$C_1=\operatorname{conv}\left\{\begin{pmatrix}2 \\
1\end{pmatrix},\begin{pmatrix}\frac{1}{2} \\
1\end{pmatrix}\right\},\
C_2=\operatorname{conv}\left\{\begin{pmatrix}-2
\\ \phantom{-}1\end{pmatrix},\begin{pmatrix}\phantom{-}0 \\
-1\end{pmatrix}\right\},$$ is an upper coexhauster of the function
$h_1$. Since $C_1\bigcap L_0=\emptyset$ (see Fig.
\ref{abbasov_ex1_epsimage} a) condition
(\ref{abbasov_th1_condition}) is not fulfilled here. Thus $h_1$ is
not bounded below (see Fig. \ref{abbasov_ex1_epsimage} b).

\begin{figure}[H]
\begin{minipage}[h]{0.49\linewidth}
\center{\includegraphics[width=0.7\linewidth]{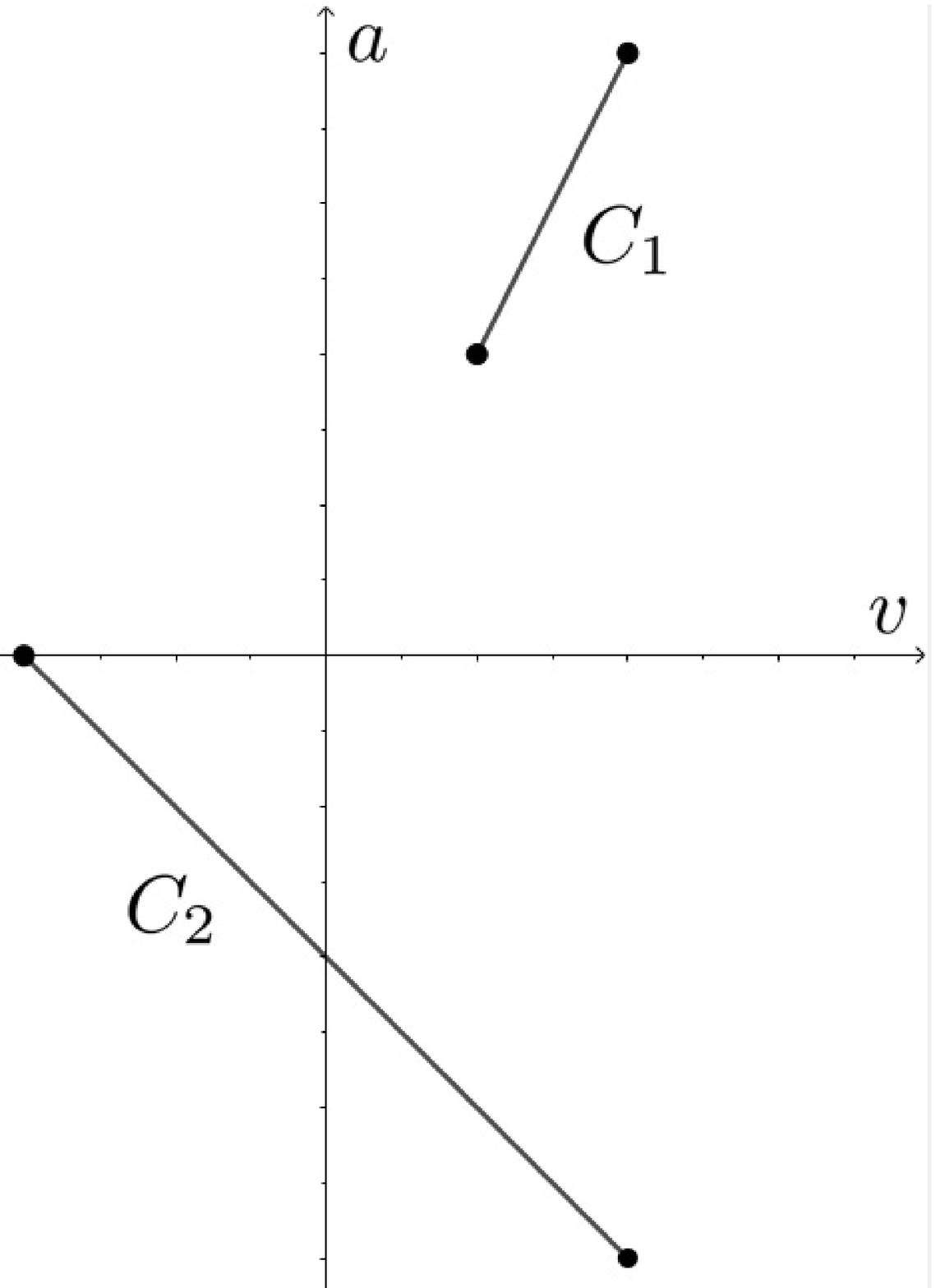} \\
a)}
\end{minipage}
\hfill
\begin{minipage}[h]{0.49\linewidth}
\center{\includegraphics[width=0.825\linewidth]{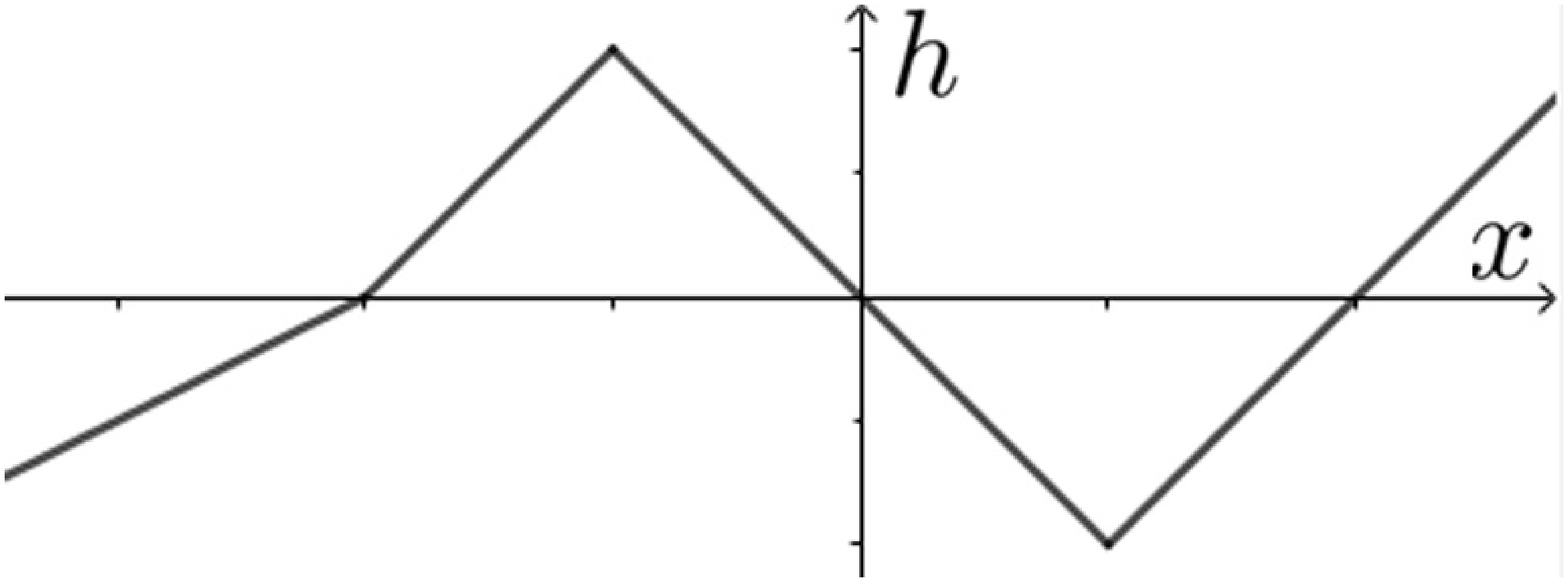} \\
b)}
\end{minipage}
\caption{The upper coexhauster and the graph of the function
$h_1$.} \label{abbasov_ex1_epsimage}
\end{figure}

\end{example}

\section{Finding the minimum value of a piecewise affine function} \label{abbasov_sec3}

On the basis of Theorems \ref{demyanov_opt_cond_main} and
\ref{abbasov_th1} we can derive a technique of finding the minimum
value of a piecewise affine function.

\begin{theorem}\label{abbasov_th2}
Let a piecewise affine function $h\colon
\mathbb{R}^n\to\mathbb{R}$ is given by
$$h(x)=\min_{C\in{\overline{E}}}\max_{[a,v]\in{C}}[a+\langle
v,x\rangle],$$ where $\overline{E}$ is a finite family of convex
polytopes. If there exists $C\in{\overline{E}}$ such that
$C\bigcap L_0=\emptyset$, then $h$ is not bounded from below on
$\mathbb{R}^n$. Otherwise
$$\displaystyle\min_{x\in\mathbb{R}^n}h(x)= a_\ast=\min_{C\in{\overline{E}}}\max_{[a,0_n]\in
C}a,$$ i.e. $a_\ast$ is the minimum value of the function $h$ on
$\mathbb{R}^n$.
\end{theorem}

\begin{proof}
First part of the theorem is a straight corollary of Theorem
\ref{abbasov_th1}. Let us show that $a_\ast$ is the minimum value
of the function $h$ on $\mathbb{R}^n$. We will consider the case
$a_\ast\geq 0$ as the opposite case can be treated similarly.
Assume that
$$a_\ast=\min_{C\in{\overline{E}}}\max_{[a,0_n]\in
C}a=\max_{[a,0_n]\in C_\ast}a.$$ Repeating the same arguments as
in the proof of sufficiency of Theorem \ref{abbasov_th1}, we can
obtain the following chain of inequalities
$$h(x)\geq h(x_\ast)\geq a_\ast \quad \forall x\in\mathbb{R}^n.$$

To show that $h(x_\ast)=a_\ast$ we assume the contrary, i.e.
$h(x_\ast)>a_\ast$. The piecewise affine function $h(x)-h(x_\ast)$
is nonnegative for any $x$ in $\mathbb{R}^n$ therefore applying
Theorem \ref{abbasov_th1} we get the condition
\begin{equation}\label{abbasov_th2_eq1}
\left(C-(h(x_\ast),0)\right)\bigcap L_0^+\neq\emptyset \quad
\forall C\in \overline{E}.
\end{equation}
From the other side, for any element of the set $C_\ast\in
\overline{E}$ of the form $[a,0_n]$ the inequality
$$a-h(x_\ast)\leq a_\ast-h(x_\ast)<0$$ holds. Hence it is valid
the condition
$$\left(C_\ast-(h(x_\ast),0)\right)\bigcap L_0^+=\emptyset$$
which contradicts (\ref{abbasov_th2_eq1}).


\end{proof}

As we see to find a minimum of a piecewise affine function one
have to calculate $\displaystyle\max_{[a,0_n]\in C}a$ for any
$C\in\overline{E}$. This can be done via solving
linear-programming problem. Consider an arbitrary $C\in
\overline{E}_\ast$. Since it is convex polytope it can be
described in the form
$C=\operatorname{conv}\left\{[a_i,v_i]\,\middle|\, i\in
I_{C}\right\}$, where $I_{C_\ast}$ is a finite index set. Then the
solution of the following linear-programming problem

\begin{equation}\label{abbasov_linprog_problem}
\begin{cases}
\max \displaystyle\sum_{i\in I_{C}}\lambda_i a_i\quad s.t.\\
\displaystyle\sum_{i\in I_{C}}\lambda_i v_i=0_n\\
\displaystyle\sum_{i\in I_{C}}\lambda_i=1\\
\lambda_i\geq 0\quad i\in I_{C}
\end{cases}
\end{equation}
gives us needed value $\displaystyle\max_{[a,0_n]\in C}a$.

\section{Finding a global minimizer of a piecewise affine function} \label{abbasov_sec4}

Now we can distinguish sets from the family $\overline{E}$ that
determine the global minimizer.

\begin{lemma}\label{abbasov_lemma1}
Let $h\colon \mathbb{R}^n\to\mathbb{R}$ be a bounded below
piecewise affine function
$$h(x)=\min_{C\in{\overline{E}}}\max_{[a,v]\in{C}}[a+\langle
v,x\rangle],$$ where $\overline{E}$ is a finite family of convex
polytopes. Then for any global minimizer $x_\ast$ of the function
$h$ on $\mathbb{R}^n$ holds the condition
$$h(x_\ast)=\min_{C\in{\overline{E}}}\max_{[a,v]\in{C}}[a+\langle
v,x_\ast\rangle]=\min_{C\in{\overline{E}_\ast}}\max_{[a,v]\in{C}}[a+\langle
v,x_\ast\rangle],$$ where
$\displaystyle\overline{E}_\ast=\left\{C\in\overline{E}\,\middle|
\, \max_{[a,0_n]}a=a_\ast\right\}$ and $\displaystyle
a_\ast=\min_{C\in{\overline{E}}}\max_{[a,0_n]\in C}a$, i. e.
global minimum can be attained on sets $C\in \overline{E}_\ast$
only.
\end{lemma}

\begin{proof}
Let the point $x_\ast$ be a global minimizer of the function $h$.
Choose an arbitrary $C\in\overline{E}$ such that
$C\notin\overline{E}_\ast$. The proof follows immediately from the
chain of inequalities
$$\max_{[a,v]\in{C}}[a+\langle v,x_\ast\rangle]\geq \max_{[a,0_n]\in{C}}[a+\langle v,x_\ast\rangle]=\max_{[a,0_n]\in{C}}a>a_\ast.$$
\end{proof}

Let us proceed to defining the set of global minimizers of the
function $h$. Consider an arbitrary $C_\ast\in \overline{E}_\ast$,
$C_\ast=\operatorname{conv}\left\{[a_i,v_i]\,\middle|\, i\in
I_{C_\ast}\right\}$, where $I_{C_\ast}$ is a finite index set.
Build the set $$X(C_\ast)=\left\{x\in\mathbb{R}^n\,\middle|\,
a_i+\langle v_i,x_\ast\rangle\leq a_\ast \right\}.$$ For any point
$x\in X(C_\ast)$ we have
$$a_\ast\leq h(x)=\min_{C\in{\overline{E}}}\max_{[a,v]\in{C}}[a+\langle
v,x_\ast\rangle]\leq \max_{[a,v]\in{C_\ast}}[a+\langle
v,x_\ast\rangle]\leq a_\ast,$$ whence $h(x)=a_\ast$, i.e. $x$ is a
global minimizer of the function $h$.

So $X=\displaystyle\bigcup_{C\in\overline{E}_\ast}X(C)$ is a set
of global minimizers of the function $h$.

Let us show that any global minimizer of the function $h$ belongs
to the set $X$. Indeed, assume that for some point
$\widehat{x}\in\mathbb{R}^n$ the equality $h(\widehat{x})=a_\ast$
is true. Then Lemma \ref{abbasov_lemma1} implies that there exists
$C_\ast\in\overline{E}_\ast$ such that
$$a_\ast=\min_{C\in\overline{E}}\max_{[a,v]\in C}[a+\langle
v,\widehat{x}\rangle]=\min_{C\in\overline{E}_\ast}\max_{[a,v]\in
C}[a+\langle v,\widehat{x}\rangle]=\max_{[a,v]\in
C_\ast}[a+\langle v,\widehat{x}\rangle],$$ which means that $x\in
X(C_\ast)$.

By these reasonings we proved the following result.

\begin{theorem}\label{abbasov_th3}
Let $h\colon \mathbb{R}^n\to\mathbb{R}$ be a bounded below
piecewise affine function
$$h(x)=\min_{C\in{\overline{E}}}\max_{[a,v]\in{C}}[a+\langle
v,x\rangle],$$ where $\overline{E}$ is a finite family of convex
polytopes. Then for a point point $x_\ast$ to be a global
minimizer of the function $h$ on $\mathbb{R}^n$ it is necessary
and sufficient that $x_\ast\in
X=\displaystyle\bigcup_{C\in\overline{E}_\ast}X(C)$.
\end{theorem}

Theorem \ref{abbasov_th3} gives important corollares.

\begin{corollary}\label{abbasov_corollary1_th3}
Let for some $C_\ast\in\overline{E}_\ast$,
$C_\ast=\operatorname{conv}\left\{[a_i,v_i]\,\middle|\, i\in
I_{C_\ast}\right\}$ the inequalities $a_i\leq a_\ast$ are valid
for any $[a_i,v_i]\in I_{C_\ast}$. Then $0_n\in X$, i.e. $0_n$ is
a global minimizer of the function $h$ on $\mathbb{R}^n$.
\end{corollary}

\begin{proof}
Since
$$a_\ast\leq h(0_n)=\min_{C\in{\overline{E}}}\max_{[a,v]\in{C}}a\leq \max_{[a,v]\in{C_\ast}}a\leq a_\ast,$$
we have $h(0_n)=a_\ast$. Hence $0_n$ is a global minimizer of $h$.
\end{proof}

\begin{corollary}\label{abbasov_corollary2_th3}
The function $h$ has a unique global minimizer $x_\ast$ iff the
set \mbox{$X=\{x_\ast\}$} is a singleton, i.e. for any
$C\in\overline{E}_\ast$ where
$C=\operatorname{conv}\left\{[a_i,v_i]\,\middle|\, i\in
I_{C}\right\}$, the system of inequalities
$$a_i+\langle v_i,x\rangle\leq a_\ast \quad i\in I_{C}.$$
is trivial in the sense that the point $x_\ast$ is its only
solution.
\end{corollary}

If our aim is to find one arbitrary global minimizer of $h$, we
can consider any $C_\ast\in\overline{E}_\ast$,
$C_\ast=\operatorname{conv}\left\{[a_i,v_i]\,\middle|\, i\in
I_{C_\ast}\right\}$ and choose at least one point satisfying the
system of inequalities
$$ a_i+\langle v_i,x\rangle\leq a_\ast\quad
i\in I_{C_\ast}.
$$
This problem can be reduced to  LP (liner programming) problem or
to unconstrained minimization of the following continuously
differentiable function
\begin{equation}\label{abbasov_funct_glob_min}
\sum_{i\in I_{C_\ast}} \left(\max\{0,a_i-a_\ast+\langle
v_i,x\rangle\}\right)^2,
\end{equation}
which gradient obviously equals
$$\sum_{i\in I_{C_\ast}} 2v_i\max\{0,a_i-a_\ast+\langle
v_i,x\rangle\}.$$

Consider some examples to demonstrate obtained results. 

\begin{example}\label{abbasov_exmpl2}
Let the function $h_2\colon\mathbb{R}\to\mathbb{R}$ is given by
$$h_2(x)=\min\left\{\max\left\{-9+2x,9-4x\right\},\max\left\{4+x,-2+\frac{1}{2}x\right\}\right\}.$$
The family $\overline{E}=\{C_1,C_2\}$, where
$$C_1=\operatorname{conv}\left\{\begin{pmatrix}-9 \\ \phantom{-}2\end{pmatrix},\begin{pmatrix}\phantom{-}9 \\ -4\end{pmatrix}\right\},
\quad C_2=\operatorname{conv}\left\{\begin{pmatrix}4 \\
1\end{pmatrix},\begin{pmatrix}-2 \\
-\frac{1}{2}\end{pmatrix}\right\},$$ is an upper coexhauster of
the function $h_2$. Since $C_i\bigcap L_0\neq\emptyset$ for all
$i=1,2$ (see Fig. \ref{abbasov_ex2_epsimage} a) condition
(\ref{abbasov_th1_condition}) is fulfilled here. Thus $h_2$ is
bounded below (see Fig. \ref{abbasov_ex2_epsimage} b).

\begin{figure}[H]
\begin{minipage}[h]{0.44\linewidth}
\center{\includegraphics[width=0.65\linewidth]{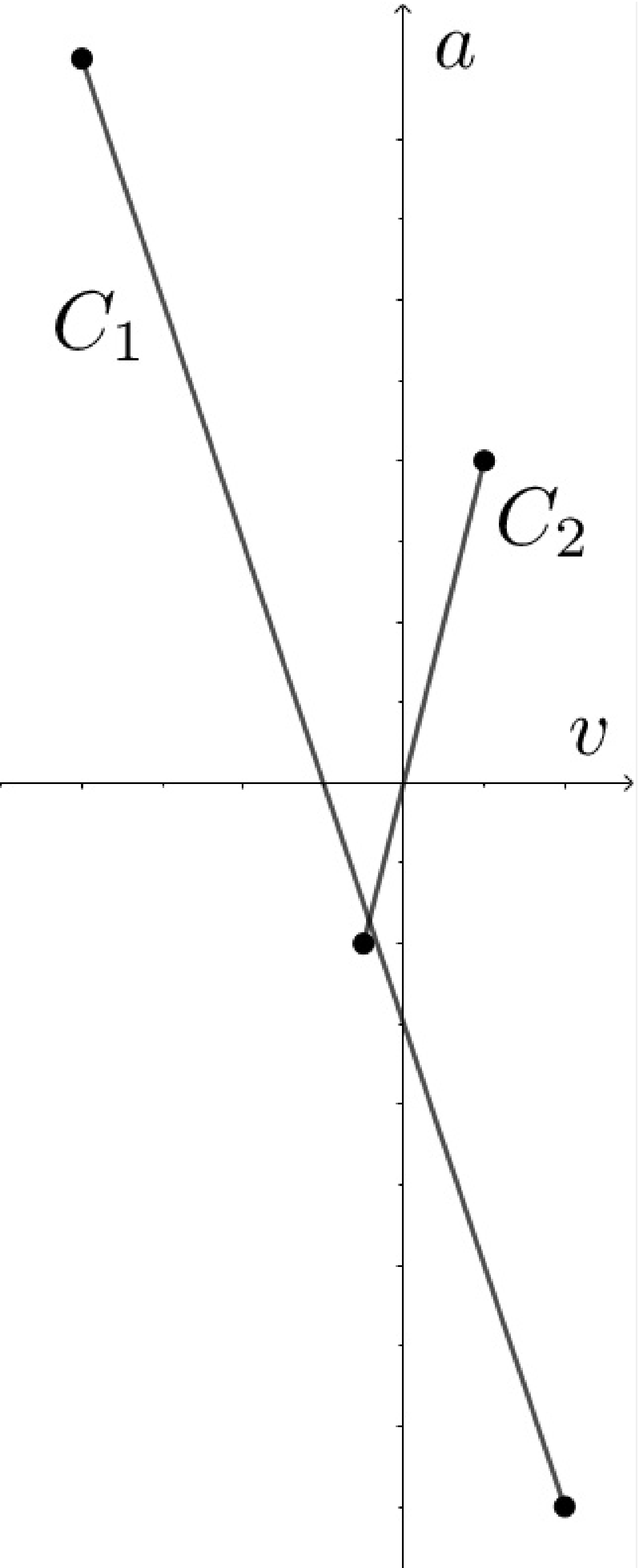} \\
a)}
\end{minipage}
\hfill
\begin{minipage}[h]{0.54\linewidth}
\center{\includegraphics[width=0.975\linewidth]{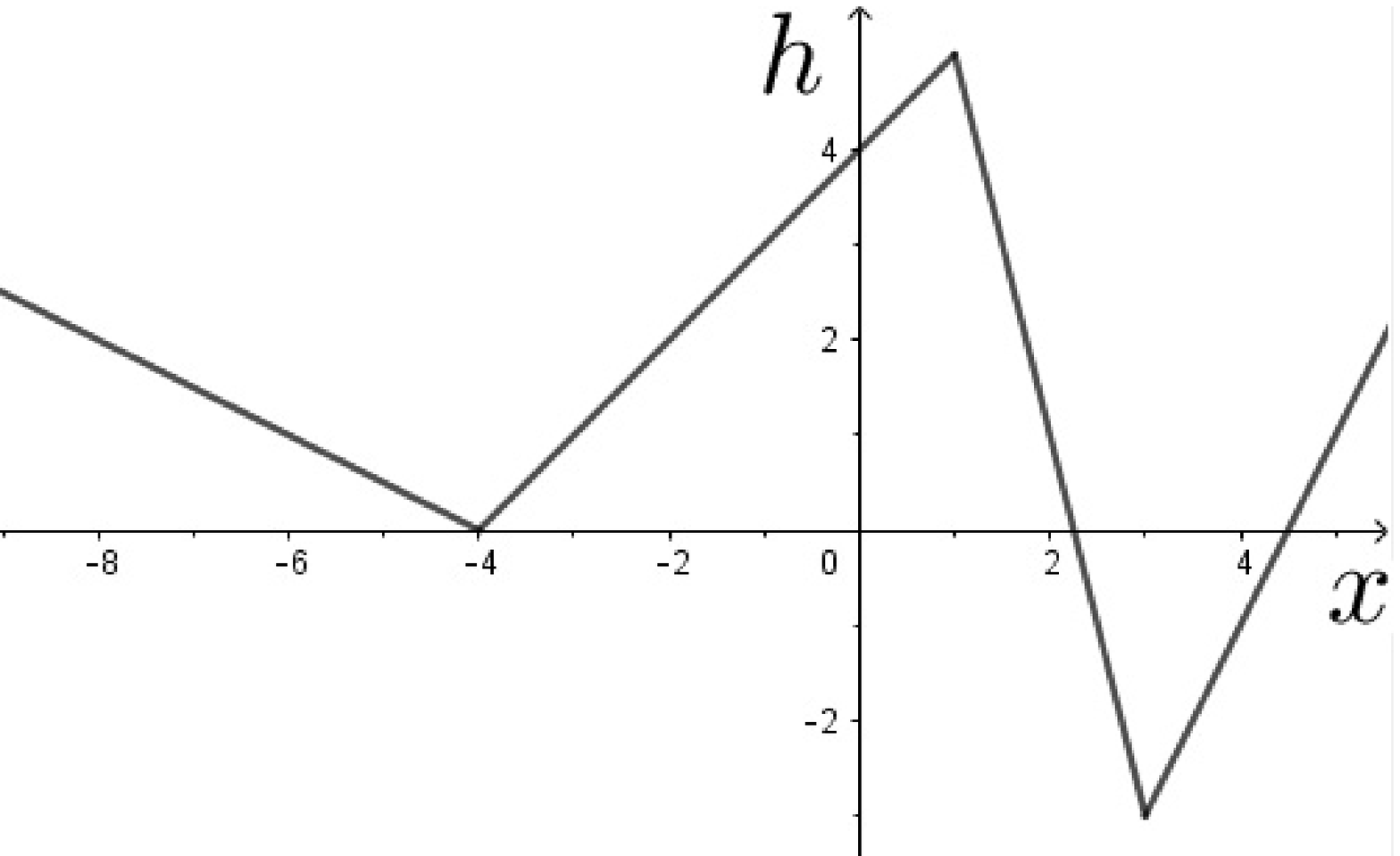} \\
b)}
\end{minipage}
\caption{The upper coexhauster and the graph of the function
$h_2$.} \label{abbasov_ex2_epsimage}
\end{figure}

It is obvious that $a_\ast=-3$ and $\overline{E}_\ast=\{C_1\}$
here. To find a set of all global minimizers we have to solve
system of inequalities
$$
\begin{cases}
-9+2x\leq -3\\
\phantom{-}9-4x\leq -3
\end{cases}
$$
which is trivial since point $x_\ast=3$ is the only solution of
the system. Hence the conditions of Corollary
\ref{abbasov_corollary2_th3} are satisfied and the function $h_2$
has only one global minimizer $x_\ast$.

We could find the unique global minimizer of the function $h_2$
via the following linear programming problem
$$
\begin{cases}
\min y \quad s.t.\\
-9+2x\leq y\\
\phantom{-}9-4x\leq y
\end{cases}
$$
which has solution $(x_\ast,y_\ast)=(3,-3)$. Or we could construct
and minimize the function (\ref{abbasov_funct_glob_min})
$$\left(\max\{0,12-4x\}\right)^2+\left(\max\{0,-6+2x\}\right)^2.$$
The gradient of this function
$$-8\max\{0,12-4x\}+4\max\{0,-6+2x\}$$
equals zero at the point $x_\ast$.

%

\end{example}

\begin{example}[see \cite{Dolgopolik2019}]\label{abbasov_exmpl3}
Consider the function $h_3\colon\mathbb{R}^2\to\mathbb{R}$ which
is given by
\begin{equation*}
\begin{split}
h_3(x)=\min\{\max\{-x_1,&x_1,-x_2,x_2\},\\&\max\{2x_1-3,-2x_1+5,x_2-1,-x_2+3\}\}.
\end{split}
\end{equation*}
The family $\overline{E}=\{C_1,C_2\}$, where
$$C_1=\operatorname{conv}\left\{\begin{pmatrix}\phantom{-}0 \\ -1\\\phantom{-}0\end{pmatrix},
\begin{pmatrix}0 \\ 1\\0\end{pmatrix},
\begin{pmatrix}\phantom{-}0 \\ \phantom{-}0\\-1\end{pmatrix},
\begin{pmatrix}0 \\ 0\\1\end{pmatrix}\right\},$$
$$C_2=\operatorname{conv}\left\{\begin{pmatrix}-3 \\ \phantom{-}2\\\phantom{-}0\end{pmatrix},
\begin{pmatrix}\phantom{-}5 \\ -2\\\phantom{-}0\end{pmatrix},
\begin{pmatrix}-1 \\ \phantom{-}0\\\phantom{-}1\end{pmatrix},
\begin{pmatrix}\phantom{-}3 \\ \phantom{-}1\\-1\end{pmatrix}\right\}.$$
is an upper coexhauster of the function $h_3$. Omitting
calculations for the problems (\ref{abbasov_linprog_problem}) for
each of sets $C_1$ and $C_2$, we write final results
$$\overline{E}_\ast=\{C_1\}, \quad a_\ast=0.$$
As we see the condition from Corollary
\ref{abbasov_corollary1_th3}
$$a\leq 0 \quad \forall [a,v]\in C_1$$
is fulfilled here, whence $x_\ast=0_n$ is a global minimizer of
$h_3$. Moreover, since the system of inequalities
$$\begin{cases}
-x_1\leq 0\\
\phantom{-}x_1\leq 0\\
-x_2\leq 0\\
\phantom{-}x_2\leq 0\\
\end{cases}
$$
is trivial this point is unique global minimizer of $h_3$, i.e.
$X=\{x_\ast\}$.

\end{example}

\begin{example}\label{abbasov_exmpl4}
Let the function $h_4\colon\mathbb{R}\to\mathbb{R}$ is given by
$$h_2(x)=\min\left\{\max\left\{-4-x,0,2+x\right\},\max\left\{2-2x,0,-9+3x\right\}\right\}.$$
The family $\overline{E}=\{C_1,C_2\}$, where
$$C_1=\operatorname{conv}\left\{\begin{pmatrix}-4 \\
-1\end{pmatrix},\begin{pmatrix}0 \\ 0\end{pmatrix},
\begin{pmatrix}2 \\ 1\end{pmatrix}\right\},
\quad C_2=\operatorname{conv}\left\{
\begin{pmatrix}\phantom{-}2 \\ -2\end{pmatrix},
\begin{pmatrix}0 \\ 0\end{pmatrix},
\begin{pmatrix}-9 \\ \phantom{-}3\end{pmatrix}\right\}$$ is an upper coexhauster of
the function $h_4$.

As we see from Fig. \ref{abbasov_ex4_epsimage} a) the function
$h_4$ is lower bounded $a_\ast=0$ and
$\overline{E}_\ast=\{C_1,C_2\}$. The set $X(C_1)$ is defined by
the system of inequalities
$$\begin{cases}
-4-x\leq 0\\
2+x\leq 0
\end{cases}
$$
while the system
$$\begin{cases}
2-2x\leq 0\\
-9+3x\leq 0
\end{cases}
$$
defines the set $X(C_2)$. Consequently $$X=X(C_1)\bigcup
X(C_2)=\left\{x\in \mathbb{R}\mid -4\leq x\leq
-2\right\}\bigcup\left\{x\in \mathbb{R}\mid 1\leq x\leq
3\right\}.$$ This corresponds to what we see on Fig.
\ref{abbasov_ex4_epsimage} b.

\begin{figure}[H]
\begin{minipage}[h]{0.34\linewidth}
\center{\includegraphics[width=0.75\linewidth]{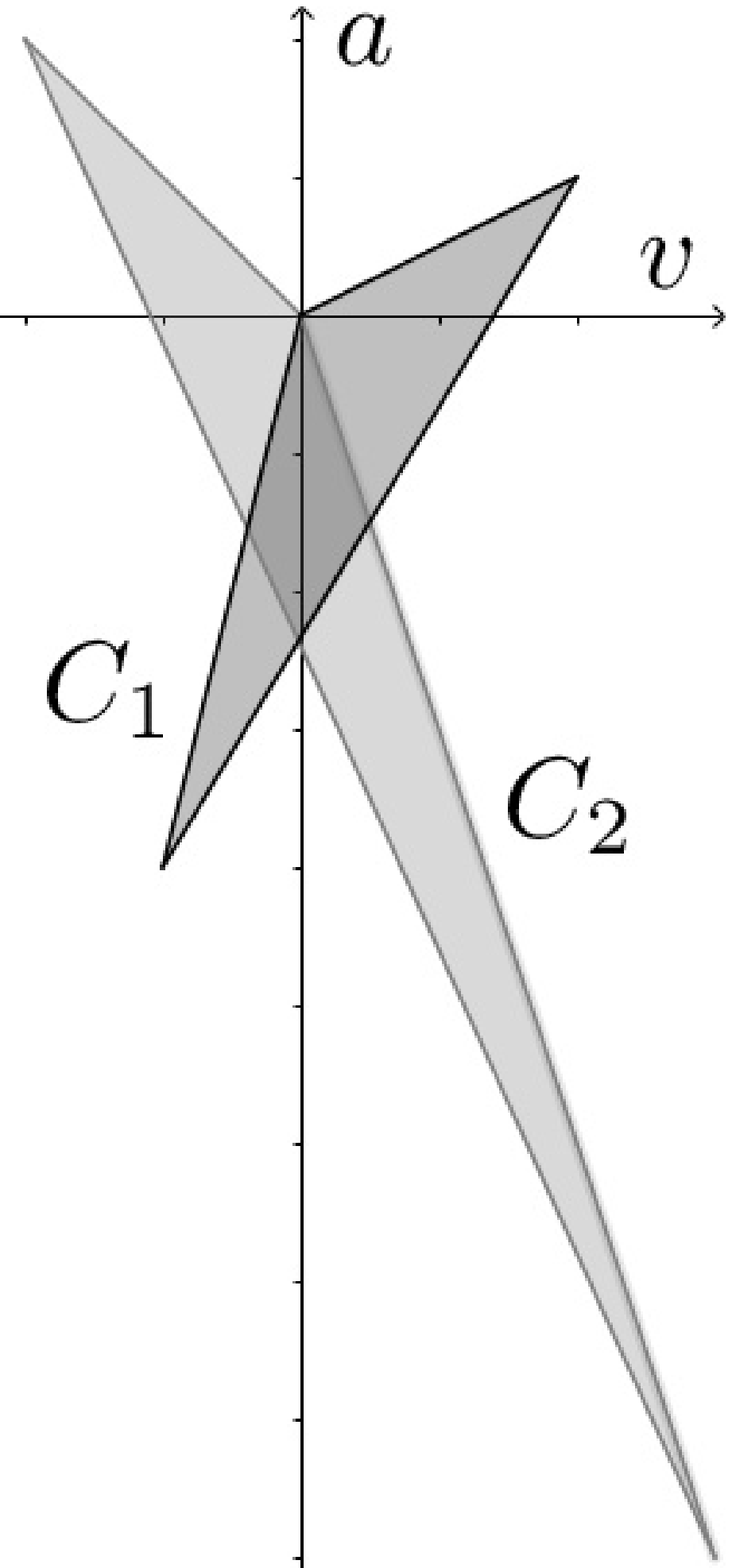} \\
a)}
\end{minipage}
\hfill
\begin{minipage}[h]{0.64\linewidth}
\center{\includegraphics[width=0.975\linewidth]{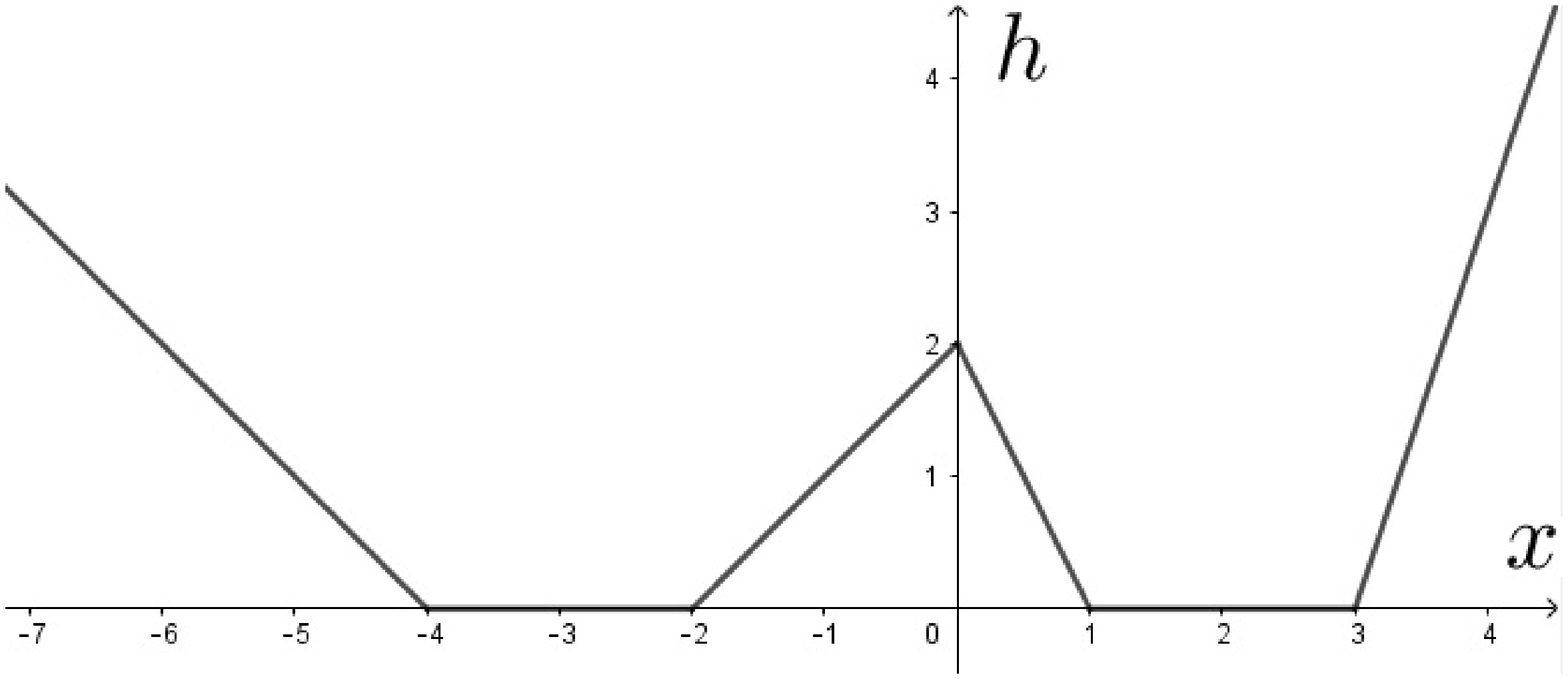} \\
b)}
\end{minipage}
\caption{The upper coexhauster and the graph of the function
$h_4$.} \label{abbasov_ex4_epsimage}
\end{figure}

\end{example}

\begin{example}\label{abbasov_exmpl5}
Consider the function $h_5\colon\mathbb{R}^{10^6}\to\mathbb{R}$
which is given by
$$h_5(x)=\min\left\{\max\left\{1+2w,-1-2w\right\},\max\left\{3,5+2w,5-2w\right\}\right\},$$
where $w=\mathbf{1}=(1,1,\dots,1)\in\mathbb{R}^{10^6}$.

The family $\overline{E}=\{C_1,C_2\}$, where
$$C_1=\operatorname{conv}\left\{\begin{pmatrix}1 \\
2w\end{pmatrix},\begin{pmatrix}-1 \\ -2w\end{pmatrix}\right\},
\quad C_2=\operatorname{conv}\left\{
\begin{pmatrix}3 \\ \mathbf{0}\end{pmatrix},
\begin{pmatrix}5 \\ 2w\end{pmatrix},
\begin{pmatrix}\phantom{-}5 \\ -2w\end{pmatrix}\right\}$$ is an upper coexhauster of
the function $h_5$.

Theorem \ref{abbasov_th1} implies that the function $h_5$ is
bounded below. It is obvious that
$$\max_{[a,0_n]\in C_1}a=0, \quad \max_{[a,0_n]\in C_2}a=5.$$
Hence $a_\ast=0$, $\overline{E}_\ast=\{C_1\}$. Consequently
$X(C_1)$ is given by the system of inequalities
$$\begin{cases}
\phantom{-}1+\langle 2w,x\rangle\leq 0\\
-1-\langle 2w,x\rangle\leq 0
\end{cases}
$$
whence $X=\left\{x\in\mathbb{R}^{10^6}\mid \langle
w,x\rangle=-\frac{1}{2} \right\}$. Based on this result we can
take the point $x=-\frac{10^{-6}}{2}w$ as a global minimizer.
\end{example}

\section*{Conclusion}
Obtained results allow one to get the minimum value of a piecewise
affine function. In comparison with the algorithm in
\cite{Dolgopolik2019} this can be done by solving $m$ linear
programming problems (\ref{abbasov_linprog_problem}) instead of
solving $m$ quadratic programming problems, where $m$ is a number
of sets in an upper coexhauster of the studied function. Note that
for a set $C\in\overline{E}$ where
$C=\operatorname{conv}\left\{[a_i,v_i]\,\middle|\, i\in
I_{C}\right\}$, the problem of this kind is
$\left|I_{C}\right|$-dimensional. So the minimum value can be
obtained easily even for high-dimensional studied function, in
cases when $\left|I_{C}\right|$ for any $C\in\overline{E}$ are
small (see Example \ref{abbasov_exmpl5}).

We can construct the set of all global minimizers of a studied
piecewise affine function via its minimum value. In cases when an
arbitrary global minimizer is needed we additionally solve linear
programming problem or the unconstrained optimization problem for
the function (\ref{abbasov_funct_glob_min}).

Proposed approach is quite simple and can be applied in different
branches of mathematics and various applications.

\section*{Acknowledgements}
Results in Section 4 were obtained in the Institute for Problems
in Mechanical Engineering of the Russian Academy of Sciences with
the support of Russian Science Foundation (RSF), project No.
20-71-10032.


%
%



\end{document}